\documentclass[11pt,amsfonts]{article}
\usepackage[utf8]{inputenc}
\usepackage[T1]{fontenc} 
\usepackage{tikz}
\usepackage[english]{babel} 
\usepackage{mathrsfs}
\usepackage{amssymb}
\usepackage[cmex10]{amsmath}
\usepackage{amsthm}
\usepackage{wrapfig}
\usepackage{latexsym}
\usepackage{amsfonts}
\usepackage{color}
\usepackage[font=small, labelfont=bf]{caption}
\usepackage{ctable}
\usepackage{floatrow}
\usepackage[colorlinks=true, citecolor=red, linkcolor=blue]{hyperref}
\usepackage{graphicx}
\usepackage{subfig}
\usepackage{verbatim}
\usepackage{epstopdf}
\usepackage{float}
\usepackage[symbol]{footmisc}
\usepackage{lipsum}

\usepackage[top=2cm , bottom=2cm , left=2.5cm ,right=2.5cm ]{geometry}

\usepackage{dsfont}

\usepackage{cases}

\theoremstyle{plain}
\newtheorem{definition}{Definition}
\newtheorem{theorem}{Theorem}

\newtheorem{proposition}{Proposition}

\newtheorem{lemma}{Lemma}

\newtheorem{remark}{Remark}

\numberwithin{equation}{section}
\numberwithin{theorem}{section}
\numberwithin{proposition}{section}
\numberwithin{definition}{section}
\numberwithin{remark}{section}
\numberwithin{corollary}{section}
\numberwithin{lemma}{section}

\usepackage{graphicx}

\makeatletter
\newcommand*\bigcdot{\mathpalette\bigcdot@{.5}}
\newcommand*\bigcdot@[2]{\mathbin{\vcenter{\hbox{\scalebox{#2}{$\m@th#1\bullet$}}}}}
\makeatother

\newcommand\blfootnote[1]{%
	\begingroup
	\renewcommand\thefootnote{}\footnote{#1}%
	\addtocounter{footnote}{-1}%
	\endgroup
}
\begin{document}
	
	\renewcommand{\thefootnote}{\arabic{footnote}}
	\begin{center}
		{\Large \textbf{ Reflected Mckean-Vlasov stochastic differential equations  with jumps in time-dependent domains }} \\[0pt]
	~\\[0pt] 

Imane Jarni\textcolor{blue}{$^{*}$}{\blfootnote{* Corresponding author.}\footnote[1]{College of Computing, Mohammed VI Polytechnic University, Morocco. E-mail: \texttt{imane.jarni@um6p.ma, jarni.imane@gmail.com}}}, Badr Missaoui{\footnote[2]{Moroccan Center of Game Theory, Mohammed VI Polytechnic University, Morocco.
				E-mail: \texttt{badr.missaoui@um6p.ma}}} and Youssef Ouknine \textcolor{blue}{\footnote[3]{Africa Business School, Mohammed VI Polytechnic University, Morocco.}\footnote[4]{Mathematics Department, Faculty of Sciences Semlalia, Cadi Ayyad University, Marrakesh, Morocco.}\footnote[5]{Hassan II Academy of Sciences and Technologies, Rabat, Morocco. E-mail:\texttt{ouknine@uca.ac.ma, youssef.ouknine@um6p.ma}}}
		\\[0pt]

		~\\[0pt]
	\end{center}
	
	\renewcommand{\thefootnote}{\arabic{footnote}}

	\begin{abstract}
		In this paper, we investigate the deterministic multidimensional Skorokhod problem with normal reflection in a family of time-dependent convex domains that are càdlàg with respect to the Hausdorff metric. We then show the existence and uniqueness of solutions to multidimensional McKean-Vlasov stochastic differential equations reflected in these time-dependent domains. Additionally, we derive stability properties with respect to the initial condition and the coefficients. Finally, we establish a propagation of chaos result.
	\end{abstract}
	
	\section{Introduction}
	
	Takana in \cite{tanaka} was the first to generalize Skorokhod's work \cite{skorokhod} from one-dimensional to multidimensional reflection problems, by constructing solutions reflected in time-independent convex domains. These results were subsequently extended to encompass a broader class of time-independent domains and driving processes (see, e.g., \cite{lion}, \cite{saicho}, \cite{slominski93}). It is worth noting that the literature primarily focuses on two main types of multidimensional reflection problems. Those with normal reflection and those with oblique reflection, the latter of which occurs when the cone of reflection differs from the cone of inward normal.
	
	In the context of time-dependent domains with oblique reflection, many works address different types of domains or trajectories (see, for instance, \cite{cost}, \cite{depuis2}, \cite{Kwon}, \cite{Lund}). Of particular note is the recent work by Nyströ and Önsko \cite{Nyst}, which extends many of the previously cited results. The authors demonstrate the existence and uniqueness of strong càdlàg solutions for obliquely reflected SDEs in non-smooth time-dependent domains.
	
	To the best of our knowledge, the case of normal reflection is less developed in the context of reflected SDEs in time-dependent domains. We quote the work by Bass and Hsu \cite{Bass}, who demonstrated the existence of a strong solution to multidimensional SDEs in smooth time-dependent domains and established weak uniqueness for this equation. However, the seminal reference in the context of normal reflection is the work by Costantini, Gobet, and El Karoui \cite{elkaroui}. In their work, the authors proved the existence and uniqueness of solutions to the Skorokhod problem in smooth time-dependent domains for a more general equation than that considered by \cite{Bass}.
	
	In all previous works, the assumptions made on the domains are geometric, particularly concerning the regularity of the domain. The one-dimensional Skorokhod problem reflected in a time-independent interval has been well investigated (see, e.g., \cite{burzdy}, \cite{jarni20}, \cite{kruk} \cite{jarni21}, \cite{slominski2010}). Typically, assumptions are made regarding the trajectory of the barriers.
	
	Klimsiak, Rozkosz, and Słomiński \cite{slom15} proved the existence and uniqueness of multidimensional BSDEs with continuous driven process with normal reflection in a time-dependent domain under conditions similar to the one-dimensional case. Specifically, they assumed that  the trajectory of the domain is  right-continuous and left-limited (càdlàg) with respect to the Hausdorff metric. They  also considered the so-called Mokobodzki condition. Additionally, we mention the work by Fakhouri, Ouknine, and Ren \cite{ouknine20}, which investigates BSDEs driven by a Brownian motion and a Poisson point process when the domain is continuous with respect to the Hausdorff metric. Furthermore, Eddahbi, Fakhouri, and Ouknine \cite{ouknine18} dealt with the general case of càdlàg driven processes and càdlàg domains with respect to the Hausdorff metric. Note that all these works about reflected BSDEs are extension of  work by Gegout-Petit and Pardoux in \cite{petit} to time-dependent domains. Also, we mention that in the context of SDEs in time-dependent domains, such an approach has never been studied in the manner of \cite{slom15}. Our goal is to study reflected McKean–Vlasov SDEs in a time-dependent domain in this manner, i.e., by considering conditions on the regularity of the trajectory of the domain instead of geometric assumptions.
	
	The McKean–Vlasov SDEs, also known as mean-field SDEs, are a class of SDEs whose coefficients depend on the law of the solution. These equations have attracted attention due to their numerous applications, particularly in the study of large-population particle systems, which are well-formulated in various scientific fields such as economics, social science, and finance (see, e.g., \cite{garnier}, \cite{lambson}, \cite{lasry}).
	
	In the context of reflected McKean–Vlasov equations, Sznitman in \cite{sznitman} was the first to study the existence and uniqueness of reflected McKean–Vlasov equations in smooth bounded time-independent domains. Recently, Adam et. al. in \cite{Adam} studied reflected McKean–Vlasov diffusions in unbounded convex time-independent domains driven by a Brownian motion, considering weaker assumptions on the coefficients of the reflected SDEs.

	Our goal in this paper is to investigate reflected McKean–Vlasov equations of Wiener–Poisson type in time-dependent domains in a similar form to \cite{slom15}. More precisely, we are given a family $\mathcal{D}=\{D_t, t\in [0,T]\}$ of time-dependent  càdlàg domains with respect to Hausdorff metric. We are interested in investigating the following system:
	$$\textbf(\Gamma)\left\{\begin{aligned}
		&X_t= X_{0}+ \int_{0}^{t} b\left(s, X_s, \mathcal{L}_{X_s}\right) \mathrm{d} s+ \int_{0}^{t} \sigma\left(s, X_s, \mathcal{L}_{X_s}\right) \mathrm{d} W_s \\
		& \quad\quad	+\int_{0}^{t}\int_{\mathbb{R}^d\backslash\{0\}}\beta\left(s,  X_{s^-}, \mathcal{L}_{X_{s^-}},z\right)  \tilde{N}(\mathrm{d} s, \mathrm{d} z)+ K_t, \quad t \in[0, T], \\
		&
		X_t\in D_t, \quad \text{for every}\; t \in[0, T], \\
		& \text{for every  càdlàg and adapted process such that} \;Z_t\in D_t,\; \int_0^t \langle X_{s} - Z_s,  \mathrm{d}K_s \rangle \leq 0\;\text{for every}\;t\geq 0
	\end{aligned}\right.,$$
	where $W$ is a Brownian motion and $\tilde{N}$ is a martingale measure in $\mathbb{R}^d$. The quantities $b,\sigma,\beta$ are measurable random maps (for more details on the assumptions regarding these quantities, see Section 3), and $\mathcal{L}_{X_t}$ means the law of $X_t$, $t\geq 0$. By a solution to reflected SDEs, we mean an adapted càdlàg couple $(X,K)$ that satisfies the system $(\Gamma)$. The process $K$ is with finite total variation $|K|$ that acts only when $X$ reaches the boundary.
	
	In \cite{ouknine20}, \cite{ouknine18}, and \cite{slom15}, the authors show that the reflected BSDEs in multidimensional domains are solved under the so-called Mokobodzki condition, which assumes the existence of a special semimartingale $A$ such that $A \in\mathring{\mathcal{D}}$. The quadratic variation of the local martingale part of $A$ is integrable, and the total variation of its finite variation part is square integrable, with $\inf_{t\in [0,T]} d(A_t,\partial D_t) > 0$. In our case, we will consider a weaker assumptions by taking $A$ to be a càdlàg optional process in $\mathcal{S}_{0,T}^2$ (see the next subsection for notations).
	
	All the papers mentioned above, about reflected SDEs, the  construction of solutions is based on the properties of deterministic mapping on the space of continuous or right continuous left limited functions. Also in this work, our approach will be  based on the deterministic Skorokhod problem, therefore, our first goal of this paper is to introduce and solve reflection problem in deterministic case  that  suits the type of reflection  in the system $(\Gamma)$.
	
	An interesting aspect in the study of Mckean-Vlasov equations is the study of Propagation of chaos property. Indeed, we know that McKean-Vlasov equations naturally arise in N-particle systems, where particles interact in a mean-field way. This interaction is reflected in the dependence of coefficients on the empirical measure of the system. For a sufficiently large \( N \) and under certain conditions, particle trajectories become asymptotically independent, even in the presence of initial correlations. This phenomenon can be seen as the law of large numbers what we commonly refer to as the Propagation of Chaos property.
	Another interesting aspect in stochastic dynamical systems theory is the stability of solutions. The property of stability implies that trajectories remain relatively unchanged under minor disturbances. Therefore, one of the goals of this paper is to study the properties of solution stability and the Propagation of Chaos.
	
	\textit{Our main contribution and organization of the paper:}

	In section \ref{sec2}, we introduce a deterministic Skorokhod problem for càdlàg inputs in a time-dependent domain that is càdlàg with respect to the Hausdorff metric. We establish the existence and uniqueness of the solution to this problem.
	
	Section \ref{sec3} is divided into four subsections. In the first subsection, we study the problem of reflection for càdlàg optional processes, then provide some estimations for solutions, which will be useful for demonstrating the existence and uniqueness of solutions to the system $(\Gamma)$. In the second subsection, we demonstrate the existence of solutions to the system $(\Gamma)$ under conditions weaker than the Mokobodzki condition. In the third subsection, we study the stability of the system $(\Gamma)$ with respect to both the initial condition and the coefficients. In the last subsection, we examine the propagation of chaos property of $(\Gamma)$. 
	
	It is important to note that our results extend the theory of reflected SDEs and remain novel even without the McKean–Vlasov component.
	
	\subsection{Notation and Preliminaries}
	We introduce the following notations, which will be used throughout this paper.
	\begin{itemize}
		\item[$\bullet$] 	$\mathcal{C}:$ The space of all bounded closed convex subsets of $\mathbb{R}^d$ with nonempty interiors endowed with the Hausdorff metric $d_H$, i.e. any $D, D^{\prime} \in \mathcal{C}$,
		$$
		d_H\left(D, D^{\prime}\right)=\max \left(\sup _{x \in D} \operatorname{d}\left(x, D^{\prime}\right), \sup _{x \in D^{\prime}} \operatorname{d}(x, D)\right),
		$$
		where $\operatorname{d}(x, D)=\inf _{y \in D}|x-y|$, with |.| is the Euclidean norm on $\mathbb{R}^d.$
		\item[$\bullet$] $\mathbb{D}([0,T],\mathbb{R}^{d})$: The space of all mapping $y:[0,T]\to \mathbb{R}^d$ with càdlàg trajectories such that $y_T=y_{T^-}$.
		\item[$\bullet$]  Let $x,y\in\mathbb{D}([0,T],\mathbb{R}^{d})$. We set $|x|_T:=\sup_{s\in[0,T]}|x_s|$. By $d_{\mathbb{S}}$, we mean the Skorokhod metric on $\mathbb{D}([0,T],\mathbb{R}^{d})$. 
		\item[$\bullet$]  Let $a\in\mathbb{D}([0,T],\mathbb{R}^{d})$, $|a|_T$ denotes the total variation on $[0,T]$.
		\item[$\bullet$]  Let $D\in\mathcal{C}$, we denote by $\partial D$ the boundary of $D$. For $x\in\partial D$, $\mathcal{N}_{x}$ denotes the set of inward normal unit vectors at $x$ i-e $\eta \in \mathcal{N}_x$ iff for any $y\in D$, $\langle x-y, \eta\rangle \leq 0$.
		\item[$\bullet$]  Let $D\in\mathcal{C}$. For every $y\in\mathbb{R}^d$, we denote by  $\Pi_{D}(y)$ the projection of $y$ on $D$.
		\item[$\bullet$] If $\mathcal{D}=\{D_t,\; t\in[0,T]\}$ is a family of time-dependent sets such that  $D_t\in\mathcal{C}$ for every $t\geq 0$, then $D_t\subset B(0,r_t)$, where  $r_t$ denotes the radius of the ball $B(0,r_t)$. For $z\in \mathbb{D}([0,T],\mathbb{R}^{d}) $,  we say that  $z\in\mathcal{D}$ if  $z_t\in D_t$  for every $t\geq 0$. We define $\mathring{\mathcal{D}}$ as the set $\{\mathring{D}_t,\; t\in[0,T]\}$, where $\mathring{D}_t$ is the interior of $D_t$.
		\item[$\bullet$] For a topological space $E$, we denote by $\mathcal{B}(E)$ its Borel  $\sigma$-algebra, and by  $\mathcal{P}(E)$, the set of probability measures on  $(E,\mathcal{B}(E))$. We define $\mathcal{P}_2(E)$ as the set of all probability measures  on ($E,\mathcal{B}(E)$) with finite second-order moments. 
		For $\mu,\nu\in\mathcal{P}_2(E)$, the second-order Wasserstein distance is defined by the formula:
		$$
		W_d(\mu, \nu)=\inf \left\{\left[\int d(x,y)^2 \pi(d x, d y)\right]^{1 / 2} ; \pi \in \mathcal{P}(E \times E) \text { with marginals } \mu \text { and } \nu\right\} \text {. }
		$$
		Note that the distance considered in $\mathbb{R}^d$ is the Euclidean norm i-e., $d(x,y)=|x-y|$ for $(x,y)\in\mathbb{R}^{2d}$.
		
		Let $t\in [0,T]$. On the space $\mathbb{D}([0,t],\mathbb{R}^{d})$, we define the distance $d(\omega_1,\omega_2):=d_t(\omega_1,\omega_2)=\sup_{s\in [0,t]}|\omega_1(s)-\omega_2(s)|$ for every $\omega_1,\omega_2\in\mathbb{D}([0,t],\mathbb{R}^{d})$. 
		
		In the following, $W_t$ denotes $W_{d_{t}}$ on the space $\mathcal{P}_2(\mathbb{D}([0,t],\mathbb{R}^{d}))$, and $W$ denotes $W_{|.|}$ on the space $\mathcal{P}_2(\mathbb{R}^{d})$.
		\item For $X\in \mathbb{D}([0,T],\mathbb{R}^{d})$, $\mathcal{L}_{X}$ denotes the law of the process $X$ on $\mathcal{P}(\mathbb{D}([0,T],\mathbb{R}^{d})$. Similarly,  $\mathcal{L}_{X_t}$ denotes the law of $X_t$ on $\mathcal{P}(\mathbb{R}^d)$, where $t\geq 0$.
		\item If $\mu \in\mathcal{P}(\mathbb{D}([0,T],\mathbb{R}^d))$, we denote by $\mu_t$ the push-forward image of the measure $\mu$ under $\Pi_t$, and by $\mu_{t^-}$ the push-forward image of the measure $\mu$ under $\Pi_{t^-}$. Here, $(\Pi_t,\Pi_{t^-}):\mathbb{D}([0,T],\mathbb{R}^{d}) \rightarrow \mathbb{R}^d\times\mathbb{R}^d \text { is defined as } \Pi_{t}(\omega)=\omega_{t}\; \text{and}\;\; \Pi_{t^-}(\omega)=\omega_{t^-}  \text { for } \omega \in \mathbb{D}([0,T],\mathbb{R}^d)$.
		\item For $a\in\mathbb{R}^d, \delta_{a}$ denotes the Dirac measure at $a$.
		\item We denote the predictable $\sigma$-algebra by $\mathcal{P}$.
		\item For $k\geq 1$, we denote by $\mathbb{H}^{2,k}$ the set of $\mathbb{R}^d$-valued, càdlàg, and adapted processes $(X_{t})_{t\geq 0}$ such that $\mathbb{E}\left(\int_{0}^T|X_s|^2\mathrm{d}s\right)<\infty.$
		\item  For two stopping times $\tau$ and $\tau^{\prime}$ such that $\tau\leq \tau^{\prime}$, we denote by $\mathcal{S}_{\tau,\tau^{\prime}}^{2}$ the complete space of $\mathbb{R}^d$-valued, càdlàg,  and adapted processes   $(X_{t})_{t\geq 0}$ such that   $||X||_{\mathcal{S}^{2}_{\tau,\tau^{\prime}}}:=\mathbb{E}(\sup_{t\in [\tau,\tau^{\prime}]}|X_{t}|^2)^{\frac{1}{2}}<\infty$.
	\end{itemize}
	\begin{remark}(\cite{Menaldi})\label{convex}
		Note that for  $D\in\mathcal{C}$, we have:
		\begin{itemize}
			\item[i)] $\forall y \in D \quad\left\langle \Pi_D(x)-x, \Pi_D(x)-y\right\rangle \leq 0 .$
			\item[ii)] If  $a \in \mathring{D}$, then for every $\eta \in \mathcal{N}_x$,
			$$
			\langle x-a, \eta\rangle \leq-\operatorname{dist}(a, \partial D).
			$$
		\end{itemize}
	\end{remark}
	\begin{proposition}(\cite{Bill})\label{bill_r}
		For each $y$ in $\mathbb{D}([0,T],\mathbb{R}^{d})$ and each positive $\epsilon$, there exist points $t_0, t_1, \ldots, t_v$ such that
		$$
		0=t_0<t_1<\cdots<t_v=T,
		$$
		and
		$$
		\sup_{t,s\in[t_{i-1}, t_i[}|y_t-y_s|<\epsilon, \quad i=1,2, \ldots, v .
		$$
	\end{proposition}
	\begin{remark}\label{rq2} The following results can be found in \cite{Bill}, and \cite{Delarue}.
		\begin{itemize}
			\item Note that $(\mathbb{D}([0,T],\mathbb{R}^{d}),d_{\mathbb{S}})$ is a  Polish space. It follows that $(\mathcal{P}_2(\mathbb{D}([0,T],\mathbb{R}^{d}),d_{\mathbb{S}}))$ is also a  Polish space.
			\item The metric $d_{\mathbb{S}}$ satisfies the following: for every   $x,y\in\mathbb{D}([0,T],\mathbb{R}^{d})$, we have $d_{\mathbb{S}}(x,y)\leq |x-y|_T$, which implies $W_{d_\mathbb{S}}\leq W_{T}$.
		\end{itemize}
	\end{remark}
	\section{Skorokhod reflection problem in a time-dependent convex domains: Deterministic case}\label{sec2}
	The Skorokhod reflection problem in the deterministic case is stated as follows:
	\begin{definition}\label{sk_det}
		Let  $\mathcal{D}=\{ D_t, t\in[0,T]\}$ be a family of time-dependent sets such that for every $t$,  $D_t\in\mathcal{C}$ and $t\to D_t$ is càdlàg with respect to Hausdorff metric. Let  $y\in\mathbb{D}([0,T],\mathbb{R}^d)$ such that
		$ y_{0}\in \mathrm{D}_0$. We say that a pair of functions $(x,k)\in\mathbb{D}([0,T],\mathbb{R}^{2d})$  is a solution to the reflection problem associated with $y$ if:
		\begin{itemize}	
			\item[(i)]  $x_t=y_t+k_t,\; \text{for every} \;t \in [0,T]$,
			\item[(ii)] $x \in \mathcal{D}$,
			\item[(iii)] For every càdlàg $z$, such that $z\in\mathcal{D}$, we have  for every $t\geq 0$
			$$\int_0^t \langle x_{s} - z_s,  \mathrm{d}k_s \rangle \leq 0.$$
		\end{itemize}
		This problem will be denoted by $RP_{\mathcal{D}}(y)$.
	\end{definition}
	\begin{remark}[\cite{slom15}]
		Note that the point (iii) in the definition \ref{sk_det} implies the existence of a unit vector $\eta$ such that $\mathrm{d}k_s=\eta_s \mathrm{d}|k|_s$ with $\eta_s\in \mathcal{N}_{x_s}$ if $x_s\in \partial D_s$.
	\end{remark}
	We establish the following a priori estimate for the solutions of problem $SP(y)$.
	\begin{lemma}\label{estim}
		Let  $\mathcal{D}=\{ D_t, t\in[0,T]\}$ and $\mathcal{\tilde{D}}=\{ \tilde{D}_t, t\in[0,T]\}$ be  two families of time-dependent sets such that for every $t\geq 0$, $D_t$ and $\tilde{D}_t$ belong to $\mathcal{C}$, $t\to D_t$ and $t\to \tilde{D}_t$ are càdlàg with respect to Hausdorff metric, $D_T=D_{T^-}$, and $\tilde{D}_T=\tilde{D}_{T^-}$. Let $y$ and $\tilde{y}$ be elements of $\mathbb{D}([0,T],\mathbb{R}^d)$ such that $y_0\in D_0$ and $\tilde{y}_0\in\tilde{D}_0$. Let $(x,k)$ and $(\tilde{x},\tilde{k})$ be the solutions associated to $SP_\mathcal{D}(y)$ and $SP_\mathcal{\tilde{D}}(\tilde{y})$, respectively.
		Then, for every $t\geq 0$, we have
		\begin{equation}
			|x_t-\tilde{x}_{t}|^2\leq |y_t-\tilde{y}_{t}|^2 + 2\sup_{s\in [0,t]} d_H(D_s,\tilde{D}_s)(|k|_t+|\tilde{k}|_t)+2\int_{0}^{t}\langle y_t-y_s+\tilde{y}_s-\tilde{y}_t,\mathrm{d}(k_s-\tilde{k}_s) \rangle.\label{ineq_estim}
		\end{equation}
	\end{lemma}
	\begin{proof}
		Let $t\geq 0$,	we apply the change of variable formula to $k-\tilde{k}$, we get:
		$$|x_t-\tilde{x}_{t}|^2\leq |y_t-\tilde{y}_{t}|^2 + 2\int_{0}^{t}\langle x_s-\tilde{x}_s,\mathrm{d}(k_s-\tilde{k}_s)\rangle+2\int_{0}^{t}\langle y_t-y_s+\tilde{y}_s-\tilde{y}_t,\mathrm{d}(k_s-\tilde{k}_s) \rangle.$$
		We have,
		\begin{align*}
			\int_{0}^{t}\langle x_s-\tilde{x}_s,\mathrm{d}(k_s-\tilde{k}_s)\rangle&= \int_{0}^{t}\langle x_s-\Pi_{D_s}(\tilde{x}_s)+\Pi_{D_s}(\tilde{x}_s)-\tilde{x}_s,\eta_s\rangle\mathds{1}_{\{x_s\in\partial D_s\}}\mathrm{d}|k|_s \\
			&- \int_{0}^{t}\langle x_s-\Pi_{\tilde{D}_s}(x_s)+\Pi_{\tilde{D}_s}(x_s)-\tilde{x}_s,\tilde{\eta}_s\rangle\mathds{1}_{\{\tilde{x}_s\in\partial \tilde{D}_s\}}\mathrm{d}|\tilde{k}|_s.
		\end{align*}
		By (iii) in Definition \ref{sk_det}, 
		$$ \int_{0}^{t}\langle x_s-\Pi_{D_s}(\tilde{x}_s),\eta_s\rangle \mathds{1}_{\{x_s\in\partial D_s\}}\mathrm{d}|k|_s\leq 0 \quad \text{and}\quad \int_{0}^{t}\langle \tilde{x}_s-\Pi_{\tilde{D}_s}(x_s),\tilde{\eta}_s\rangle \mathds{1}_{\{\tilde{x}_s\in\partial \tilde{D}_s\}}\mathrm{d}|\tilde{k}|_s\leq 0.$$
		It follows that,
		\begin{align*}
			\int_{0}^{t}\langle x_s-\tilde{x}_s,\mathrm{d}(k_s-\tilde{k}_s)\rangle& \leq \int_{0}^{t}\langle \Pi_{D_s}(\tilde{x}_s)-\tilde{x}_s,\eta_s\rangle\mathds{1}_{\{x_s\in\partial D_s\}}\mathrm{d}|k|_s+  \int_{0}^{t}\langle \Pi_{\tilde{D}_s}(x_s)-x_s,\tilde{\eta}_s\rangle\mathds{1}_{\{\tilde{x}_s\in\partial \tilde{D}_s\}}\mathrm{d}|\tilde{k}|_s\\
			&\leq \int_{0}^{t} |\Pi_{D_s}(\tilde{x}_s)-\tilde{x}_s|\mathds{1}_{\{x_s\in\partial D_s\}}\mathrm{d}|k|_s + \int_{0}^{t} |\Pi_{\tilde{D}_s}(x_s)-x_s|\mathds{1}_{\{x_s\in\partial \tilde{D}_s\}}\mathrm{d}|\tilde{k}|_s\\
			&\leq  \int_{0}^{t} d(\tilde{x}_s,D_s)\mathrm{d}|k|_s + \int_{0}^{t} d(x_s,\tilde{D}_s)\mathrm{d}|\tilde{k}|_s\\
			&\leq \sup_{s \in[0,t]} d_H(D_s,\tilde{D}_s)(|k|_t+|\tilde{k}|_t).
		\end{align*}
		Subsequently, we obtain:
		$$|x_t-\tilde{x}_{t}|^2\leq |y_t-\tilde{y}_{t}|^2 + 2\sup_{s\in [0,t]} d_H(D_s,\tilde{D}_s)(|k|_t+|\tilde{k}|_t)+2\int_{0}^{t}\langle y_t-y_s+\tilde{y}_s-\tilde{y}_t,\mathrm{d}(k_s-\tilde{k}_s) \rangle.$$
		This completes the proof.
	\end{proof}
	\begin{proposition}\label{est_var}
		Let	$\mathcal{D}^n=\{ D_t^n, t\in[0,T]\}$ and $\mathcal{D}=\{ D_t, t\in[0,T]\}$ be families of time-dependent sets such that for every $t\geq 0$ and $n\in\mathbb{N}$, $D_t^n,D_t\in\mathcal{C}$, and $t\to D_t^n$, $t\to D_t$ are càdlàg with respect to Hausdorff metric, $D_T^n=D_{T^-}^n$, $D_T=D_{T^-}$. Let $\left\{y^n\right\}$ and $\left\{a^n\right\}$ be elements of $\mathbb{D}\left([0,T], \mathbb{R}^d\right)$ such that $y_0^n \in D_0^n$ and $a^n\in \mathring{\mathcal{D}^n}$. Let $\left\{\left(x^n, k^n\right)\right\}$ be a sequence of solutions of the Skorokhod problem $SP_{\mathcal{D}^n}(y^n)$. We assume that $\inf_{t\in[0,T]}d(a^n_t,\partial D_t^n)\geq m$, where $m$  is a positive constant independent of $n$, and $\lim\limits_{n}(\sup _{t\in [0,T]} d_H(D_t^n, D_t)+ \sup _{t\in [0,T]}|y_t^n- y_t|+ \sup _{t\in [0,T]}|a_t^n- a_t|)=0$. Then, there exists $C>0$ and $N\in\mathbb{N}$ such that for every $n\geq N$, we have
		\begin{equation}
			|k^n|_{T} \leq C( \sup_{t\in [0,T]}|y_t|^2+ \sup_{t\in [0,T]}|a_t|^2 +1). \label{estim_k_n}
		\end{equation}
	\end{proposition}
	\begin{proof}
		Since $y^n$ and $a^n$ converge uniformly to $y$ and $a$, respectively, and $a^n\in\mathcal{D}$, it follows that $y$ and $a$ are càdlàg, and $a\in\mathcal{D}$. By proposition \ref{bill_r}, there exist  points $t_0, t_1, \ldots, t_v$ such that 
		$ 0=t_0<t_1<\cdots<t_v=T$ and
		\begin{equation}
			\sup_{t,s\in[t_{i-1}, t_i[}|a_t-a_s|+\sup_{t,s\in[t_{i-1}, t_i[}|y_t-y_s|<\frac{m}{4}, \quad i=1,2, \ldots, v. \label{ineq_bill}
		\end{equation}
		Let $i\in\{1\dots v\}$ and $t\in[t_{i-1},t_i]$. We apply the change of variable formula  to the function $k_.^n-k_{t_{i-1}}^n$, we get
		\begin{align*}
			|k_t^n-k_{t_{i-1}}^n|^2&\leq 2 \int_{t_{i-1}}^t\langle x_u^n-x_{t_{i-1}}^n,\mathrm{d}k_u^n\rangle - 2\int_{t_{i-1}}^t\langle y_u^n-y_{t_{i-1}}^n,\mathrm{d}k_u^n\rangle\\ 
			&\leq 2\int_{t_{i-1}}^t\langle x_u^n-a_u^n,\mathrm{d}k_u^n\rangle + 2\int_{t_{i-1}}^t\langle a_u^n-a_{u},\mathrm{d}k_u^n\rangle+ 2\int_{t_{i-1}}^t\langle a_u-a_{t_{i-1}},\mathrm{d}k_u^n\rangle\\
			& + 2\int_{t_{i-1}}^t\langle a_{t_{i-1}}-x_{t_{i-1}}^n,\mathrm{d}k_u^n\rangle - 2\int_{t_{i-1}}^t\langle y_u^n-y_{u},\mathrm{d}k_u^n\rangle- 2\int_{t_{i-1}}^t\langle y_u-y_{t_{i-1}},\mathrm{d}k_u^n\rangle\\
			& - 2\int_{t_{i-1}}^t\langle y_{t_{i-1}}-y_{t_{i-1}}^n,\mathrm{d}k_u^n\rangle,\\
		\end{align*}
		where in the last inequality, we have add and subtract $a^n$, $a-a_{t_{i-1}}$ and $y-y_{t_{i-1}}$. 
		
		Now, we use (ii) of Remark \ref{convex}, and we apply the Cauchy-Schwartz inequality to the second, fifth, and seventh integrals. We split the third and sixth integrals into integrals over the intervals $[t_{i-1}, t[$ and $\{t\}$. In the fourth integral, we add and subtract $\Pi_{D_{t_{i-1}}}(x^n_{t_{i-1}})$. By doing so, we obtain
		\begin{equation*}
			\begin{split}
				|k_t^n-k_{t_{i-1}}^n|^2&\leq -  2\int_{t_{i-1}}^t d(a_u^n,\partial D_u^n)\mathrm{d}|k^n|_u + 2 \sup_{s\in [0,T]}|a_s^n-a_s||k^n|_{[t_{i-1},t_{i}]}+ 2 \int_{[t_{i-1},t[}\langle a_u-a_{t_{i-1}},\mathrm{d}k_u^n\rangle\\
				&+2\langle a_t-a_{t_{i-1}}, \Delta k_{t}^n \rangle+2\langle a_{t_{i-1}}-\Pi_{D_{t_{i-1}}}(x^n_{t_{i-1}}),k_t^n-k_{t_{i-1}}^n \rangle + 2\langle \Pi_{D_{t_{i-1}}}(x^n_{t_{i-1}}) -x_{t_{i-1}}^n,k_t^n-k_{t_{i-1}}^n\rangle\\
				& +4\sup_{s\in [0,T]}|y_s^n-y_s||k^n|_{[t_{i-1},t_{i}]} - 2\int_{[t_{i-1},t[}\langle y_u-y_{t_{i-1}},\mathrm{d}k_u^n\rangle -2\langle y_t-y_{t_{i-1}}, \Delta k_{t}^n \rangle\\
				&\leq -  2\int_{t_{i-1}}^t d(a_u^n,\partial D_u^n)\mathrm{d}|k^n|_u +  2 \sup_{s\in [0,T]}|a_s^n-a_s||k^n|_{[t_{i-1},t_{i}]}+ 2\sup_{t\in[t_{i-1,t_i}[}|a_t-a_{t_{i-1}}||k^n|_{[t_{i-1},t_{i}]} \\
				&+ 8 \sup_{s\in [0,T]}|a_s|\sup_{t\in[t_{i-1},t_{i}]}|k_t^n-k_{t_{i-1}}^n|+ 4r_{t_{i-1}}\sup_{t\in[t_{i-1},t_{i}]}|k_t^n-k_{t_{i-1}}^n|+2 \sup_{t\in [0,T]}d_H(D_t^n,D_t)|k^n|_{[t_{i-1},t_{i}]}\\
				&  +4\sup_{s\in[0,T]}|y_s^n-y_s||k^n|_{[t_{i-1},t_{i}]} + 2\sup_{t\in[t_{i-1,t_i}[}|y_t-y_{t_{i-1}}||k^n|_{[t_{i-1},t_{i}]}+ 8 \sup_{s\in [0,T]}|y_s|\sup_{t\in[t_{i-1},t_{i}]}|k_t^n-k_{t_{i-1}}^n|.
			\end{split}
		\end{equation*}
		According to the assumptions, there exists $N$ such that for every $n\geq N$, 
		$$2\sup_{s\in[0,T]}|a_s^n-a_s|+2\sup_{t\in[0,T]}d_H(D_t^n,D_t)+ 4\sup_{s\in[0,T]}|y_s^n-y_s|\leq \frac{m}{2}.$$
		Combining this with  \eqref{ineq_bill}, we obtain,
		\begin{multline*}
			2\int_{t_{i-1}}^t d(a_u^n,\partial D_u^n)\mathrm{d}|k^n|_u+|k_t^n-k_{t_{i-1}}^n|^2\leq\\ m|k^n|_{[t_{i-1},t_{i}]}+ 8(\sup_{s\in[0,T]}|a_s|+\sup_{s\in[0,T]}|y_s|+ +\frac{1}{2}r_{t_{i-1}})\sup_{t\in[t_{i-1},t_{i}]}|k_t^n-k_{t_{i-1}}^n|.
		\end{multline*}
		We use the inequality $2|ab|\leq (\frac{1}{2}|a|^2+2|b|^2)$, we get
		\begin{equation*}
			\begin{split}
				m|k^n|_{[t_{i-1},t_{i}]} +	\sup_{t\in[t_{i-1},t_{i}]} |k_t^n-k_{t_{i-1}}^n|^2&\leq \frac{1}{2}\sup_{t\in[t_{i-1},t_{i}]} |k_t^n-k_{t_{i-1}}^n|^2 + 64(\sup_{s\in[0,T]}|a_s|^2+\sup_{s\in[0,T]}|y_s|^2+\frac{1}{4}r_{t_{i-1}}^2).
			\end{split}
		\end{equation*}
		By setting  $c=\max\Big(64,\frac{1}{4}\max_{\{i=1,\dots v+1\}}(r_{t_{i-1}}^2)\Big)$,  we obtain
		\begin{equation*}
			|k^n|_{[t_{i-1},t_{i}]}\leq  c(\sup_{s\in[0,T]}|y_s|^2+ \sup_{s\in[0,T]}|a_s|^2 +1).
		\end{equation*}
		We sum on $i$, which is a finite sum, we get
		\begin{equation*}
			|k^n|_{T} \leq C(\sup_{s \in[0,T]}|y_s|^2 + \sup_{s \in[0,T]}|a_s|^2 +1),
		\end{equation*}
		where $C=c(v+1)$. 
		
		This completes the proof.
	\end{proof}
	\begin{lemma}\label{disc}
		Let  $\mathcal{D}=\{ D_t, t\in[0,T]\}$ be a family of time-dependent sets such that for every $t$,  $D_t\in\mathcal{C}$ and $t\to D_t$ is càdlàg with respect to Hausdorff metric such that $D_T=D_{T^-}$. Let $y$ and $a$ be elements of  $\mathbb{D}([0,T],\mathbb{R}^d)$. For  $n\geq 0$, we consider the following recursive sequence of times defined by:
		$t_0^n=0$ and  for $i\geq 1$,
		$$t_i^n=(t_{i-1}^n+1/n) \land \inf\{ t>t_{i-1}^n \quad  d_H(D_{t_{i-1}^n}, D_t) \geq \frac{1}{n} \;\;\text{or}\;\; |y_{t_{i-1}^n}-y_t|\geq \frac{1}{n}\;\;\text{or}\;\; |a_{t_{i-1}^n}-a_t|\geq \frac{1}{n}\}\land T.$$
		Then, we have
		\begin{center}
			For every $n\geq 1$, there exists $i_n$ such that for every $i>i_n$ $t_{i}^n=T$.
		\end{center}
		In addition, if we consider the  discretization of $y$, $a$ and $\mathcal{D}$ given by the following:
		$$D_t^n= \begin{cases}D_{t_{i-1}^n}, & t \in[t_{i-1}^n, t_i^n[, i=1, \ldots, i_n, \\ 
			D_{t_{i_{n}}^n}, & t \in\left[t_{i_{n}}^n, T\right] .\end{cases}$$
		$$y^n_t= \begin{cases}y_{t_{i-1}^n}, & t \in[t_{i-1}^n, t_i^n[, i=1, \ldots, i_n, \\ 
			y_{t_{i_{n}}^n}, & t \in\left[t_{i_{n}}^n, T\right] .\end{cases}$$
		$$a^n_t= \begin{cases}a_{t_{i-1}^n}, & t \in[t_{i-1}^n, t_i^n[, i=1, \ldots, i_n, \\ 
			a_{t_{i_{n}}^n}, & t \in\left[t_{i_{n}}^n, T\right] .\end{cases}$$
		We have,
		$$\lim\limits_{n}\sup _{0\leq t \leq T} d_H(D_t^n, D_t)=0, \;\;
		\lim\limits_{n}\sup _{0\leq t \leq T} |y_t^n- y_t|=0, \;\;\text{and}\;\;\lim\limits_{n}\sup _{0\leq t \leq T} |a_t^n- a_t|=0.$$
	\end{lemma}
	\begin{proof}
		\begin{itemize}
			\item[i)] The first point follows immediately from the fact that $y,a $ and $\mathcal{D}$ are càdlàg.
			\item[ii)] Assume that there exists a sequence $(t_n)$ in $[0,T]$ such that $\lim\limits_{n} t_n=t$ and $\lim\limits_{n}d_H(D_{t_n}^n,D_{t_n})\neq 0$.
			Setting $s_n=\max\{t_i^n , t_i^n\leq t_n\}$ implies that, for every $n\geq 1$, there exists $j_n$ such that $s_n=t_{j_n}^n$, 
			$D_{t_n}^n=D_{s_n}$ and $s_n\leq t_n< t_{{j_n}+1}^n\leq s_n+\frac{1}{n}$,
			hence $\lim\limits_{n} s_n=t$.
			We distinguish the following cases:
			\begin{itemize}
				\item[\textbf{case 1:}] If $\{n: t<s_n\}$ is infinite, we have
				$$0\leq d_H(D_{t_n}^n,D_{t_n})=d_H(D_{s_n},D_{t_n})\leq d_H(D_{s_n},D_{t})+ d_H(D_{t},D_{t_n}).$$
				The last terms in the previous inequality tend to 0, which is a contradiction.
				\item[\textbf{case 2:}]  If $\{n: s_n\leq t\leq t_n\}$ is infinite, then
				$$0\leq d_H(D_{t_n}^n,D_{t_n})\leq d_H(D_{s_n},D_{t})+ d_H(D_{t},D_{t_n}) \leq \frac{1}{n} +d_H(D_{t},D_{t_n}),$$
				which also tend to 0. It is also a contradiction.
				\item[\textbf{case 3:}]  If $\{n: t_n< t\}$ is infinite, then
				$$0\leq d_H(D_{t_n}^n,D_{t_n})\leq d_H(D_{s_n},D_{t^-})+ d_H(D_{t^-},D_{t_n}),$$
				which is also a contradiction.
				\item[iii)] Using the fact that $y$ and $a$ are càdlàg, the convergence of $y^n$ and $a^n$ follows similarly as (ii).
			\end{itemize}
			
		\end{itemize}
	\end{proof}
	\begin{theorem}
		Let  $\mathcal{D}=\{ D_t, t\in[0,T]\}$ be a family of time-dependent sets such that for every $t$,  $D_t\in\mathcal{C}$,  and $t\to D_t$ is càdlàg with respect to  Hausdorff metric. Let $y$ and $a$ be elements of  $\mathbb{D}([0,T],\mathbb{R}^d)$ such that $ y_{0}\in D_0$, $a\in \mathring{\mathcal{D}}$, and $\inf_{t\in[0,T]}d(a_t,\partial D_t)>0$.
		Then, there exists a unique solution of the problem $SP_{\mathcal{D}}(y)$. In addition,
		
		\begin{equation}
			|k|_{T} \leq C( \sup_{s\in [0,T]}|y_s|^2+ \sup_{s\in [0,T]}|a_s|^2 +1). \label{estim_k_}
		\end{equation}
	\end{theorem}
	\begin{proof}
		\begin{itemize}
			\item [i)] \textbf{Uniqueness:} It follows immediately from  Lemma \ref{estim}.
			\item [ii)] \textbf{Existence:}
			We consider the same discretization of $y$, $a$ and $\mathcal{D}$ as in Lemma \ref{disc} along with the following scheme.
		$$
		\begin{aligned}
			& x_t^n= \begin{cases}y_0, & \text{if} \quad t \in\left[0, t_1^n\right[, \\
				\Pi_{D_{t_{i-1}^n}^n}(x_{t_{i-2}^n}^n+y_t^n-y_{t_{i-2}^n}), & \text{if} \quad t \in[t_{i-1}^n, t_i^n[, i=1, \ldots, i_n,\\
				\Pi_{D_{t_{i_{n}-1}^n}^n}(x_{t_{i_{n}-1}^n}^n+y_t^n-y_{t_{i_{n}-1}^n}), & \text{if}\quad t \in\left[t_{i_{n}}^n, T\right]\end{cases} \\
			&
		\end{aligned}
		$$
		$$
		\begin{aligned}
			& k_t^n= \begin{cases}0, & \text{if} \quad  t \in\left[0, t_1^n[\right. \\
				k_{t_{i-2}^n}^n+x_{t}^n-x_{t_{i-2}^n}^n-y_t^n+y_{t_{i-2}^n}, & \text{if} \quad  t \in\left[t_{i-1}^n, t_i^n[\right. \, i=1, \ldots, i_n,\\
				k_{t_{i_{n}-1}^n}^n+x_t^n-x_{t_{i_{n}-1}^n}^n-y_t^n+y_{t_{i_{n}-1}^n}, & \text{if} \quad  t \in[t_{i_{n}}^n, T].\end{cases} \\
			&
		\end{aligned}
		$$ We will complete the proof in two steps.
			\begin{itemize}
				
				\item[\textbf{Step 1}:] We show that $(x^n,k^n)=SP_{\mathcal{D}^n}(y^n)$, where  $\mathcal{D}^n=\{D^n_t,\quad t \in[0, T] \}$.
				\begin{itemize}
					\item[$\bullet$] By the construction of the scheme above, we have for every $t\in [0,T]$,$$x^n_t=y^n_t+k^n_t \in D_{t}^n.$$
					\item[$\bullet$] Let $(v^n)_{n\geq 1}$ be a sequence of elements of $\mathbb{D}([0,T],\mathbb{R}^d)$ such that for every $n\geq 0$, $v^n\in\mathcal{D}^n$, then
					\begin{align*}
						\int_{0}^t \langle x^n_s-v_s^n, \mathrm{d}k^n_s\rangle=\sum_{i=1}^{i_n} \langle x_{{t_i^n}\land t}^n-v_{{t_i^n}\land t}^n, \Delta k^n_{{t_i^n}\land t}\rangle
						\\
						=\sum_{i=1}^{i_n} \langle \Pi_{D_{{t_{i}^n}\land t}^n}(x_{{{t_{i-1}^n}\land t}}^n+ \Delta y_{{t_i^n}\land t}^n)-&v_{{t_i^n}\land t}^n, \Pi_{D_{{t_i^n}\land t}^n}(x_{{{t_{i-1}^n}\land t}}^n+ \Delta y_{{t_i^n}\land t}^n)-(x_{{{t_{i-1}^n}\land t}}^n+ \Delta y_{{t_{i-1}^n}\land t}^n)\rangle.
					\end{align*}
					By (i) in Remark \eqref{convex}, the previous sum is non positive.
				\end{itemize}
				\item[\textbf{Step 2}:] We show that $x^n,k^n$ converge to the solution of the reflection problem $RP_{\mathcal{D}}(y)$ in the space $\big(\mathbb{D}([0,T],\mathbb{R}^d), |.|_T\big)$. 
				
				On the one hand, we have 
				$$\inf_{t\in[0,T]} d(a_t^n,\partial D_t^n)=\min_{i=0,\dots,i_n-1}\left(\inf_{t\in[t_{i}^n,t_{i+1}^n[}d(a_{t_{i}^n},\partial D_{t_{i}^n})\right)\land d(a_T,\partial D_T)\geq \inf_{t\in[0,T]} d(a_t,\partial D_t), $$
				which is positive and independent of $n$. Combined with  Lemma \ref{disc}, the assumptions of Proposition \ref{est_var} are satisfied. Therefore, $|k^n|$ satisfies inequality \eqref{estim_k_n}. On the other hand, combining \eqref{estim_k_n} with the inequality \eqref{ineq_estim} in Lemma \ref{estim}, we obtain $x^n$ and $k^n$ are Cauchy sequences in $\big(\mathbb{D}([0,T],\mathbb{R}^d), |.|_T\big)$. Again from Proposition \ref{est_var},  $k^n$ converges uniformly to a function with finite variation. Denote by $x$ and $k$ the limits of $x^n$ and $k^n$, respectively. 
				\begin{itemize}
					\item[$\bullet$] Let $t\in[0,T]$, we have
					$$d(x_t,D_t)\leq |x_t-x_t^n| + d(x_t^n, D_t)  \leq \sup_{t\in [0,T]}|x_t-x_t^n|+ \sup_{t\in [0,T]}d_H(D_t^n,D_t). $$
					We tend n to $\infty$, we get $d(x_t,D_t)=0$, this implies $x_t\in D_t $.
					\item[$\bullet$] We show (iii) of Definition \ref{sk_det}.
					
					Let $v\in\mathbb{D}([0,T],\mathbb{R}^d)$ such that $v\in \mathcal{D}$.
					We have $x^n$ and $k^n$ converge uniformly to $x$ and $k$.
					On the one hand, using Lemma 3.9 in \cite{ouknine20}, for every $t\geq 0$, we have
					$$\lim\limits_{n}\int_{0}^t \langle x^n_s-v_s, \mathrm{d}k^n_s\rangle= \int_{0}^t \langle x_s-v_s, \mathrm{d}k_s\rangle,$$
					on the other hand,
					\begin{align*}
						\int_{0}^t \langle x^n_s-v_s, \mathrm{d}k^n_s\rangle&=\sum_{i=0}^{i_n} \langle x_{{t_i^n}\land t}^n-v_{{t_i^n}\land t},\Delta k_{{t_i^n}\land t}^n \rangle\leq 0.
					\end{align*}
				\end{itemize}
				It is  clear that $(x,k)$ satisfy the assumptions of Proposition \ref{estim_k_n}, hence \eqref{estim_k_} holds true.
				
				This  completes the proof.
			\end{itemize}
		\end{itemize}
	\end{proof}
	
	\section{Reflected McKean-Vlasov stochastic differential equation}\label{sec3}
	In this section, we consider a probability space $\left(\Omega, \mathbb{F},\left(\mathcal{F}_t\right)_{t \leq T},\mathbb{P}\right)$ satisfying the usual conditions, with an  $\mathbb{F}$-adapted and $\mathbb{R}^m$-valued Brownian motion $(W_t)_{0 \leq t \leq T}$ and $(\tilde{N}_t)_{0 \leq t \leq T}$ is a martingale measure in $\mathbb{R}^d \backslash\{0\}$ with Lévy measure $\lambda$ independent of $(W_t)_{0 \leq t \leq T}$, corresponding to a standard Poisson random measure $N(t, A)$. For any measurable subset of $\mathbb{R}^d \backslash\{0\}$ such that $\lambda(A)<\infty$, we have:
	$$
	\tilde{N}_t(A)=N(t, A)-t \lambda(A),
	$$
	where $N(t, A)$ satisfies that
	$$
	\mathbb{E}[N(t, A)]=t \lambda(A) .
	$$
	
	The measure $\lambda(\cdot)$ is supposed to be a $\sigma$-finite measure on $\mathbb{R}^d \backslash\{0\}$ and satisfies the following integrability condition
	$$
	\int_{\mathbb{R}^d \backslash\{0\}}\left(1 \wedge|z|^2\right) \lambda(\mathrm{d} z)<\infty .
	$$
	\subsection{Application to stochastic processes}
	Now we give the definition of Skorokhod problem for processes.
	\begin{definition}\label{stoch_def}
		Let  $\mathcal{D}=\{ D_t, t\in[0,T]\}$ be a family of time-dependent sets such that for every $t$,  $D_t\in\mathcal{C}$ and $t\to D_t$ is càdlàg with respect to Hausdorff metric, and  $D_T=D_{T^-}$. Let $Y$ be a càdlàg $\mathbb{F}$-adapted process such that $Y_0\in D_0$. We say that a pair of processes $(X,K)$ is a solution of the reflection problem associated with $Y$  and $\mathcal{D}$ if $X$ and $K$ are càdlàg and $\mathbb{F}$-adapted processes, such that
		\begin{itemize}	
			\item[(i)]  $X_t=Y_t+K_t,\; \text{for every} \;t \in [0,T]$,
			\item[(ii)] $X_t \in \mathrm{D}_t, \; \text{for every} \; t \in [0,T]$,
			\item[(iii)] For every  càdlàg and $\mathbb{F}$-adapted process  $Z$ such that $Z\in \mathcal{D}$, we have 
			$$\text{for every } t\geq 0\;\;\;\int_0^t \langle X_{s} - Z_s,  \mathrm{d}K_s \rangle \leq 0.$$	
		\end{itemize}
		We write $(X,K)=SP_{\mathcal{D}}(Y)$.
	\end{definition}
	\begin{proposition}\label{sk_stoch}
		Let  $\mathcal{D}=\{ D_t, t\in[0,T]\}$ be a family of time-dependent sets such that for every $t$,  $D_t\in\mathcal{C}$ and $t\to D_t$ is càdlàg with respect to Hausdorff metric. Let $Y$ be a càdlàg $\mathbb{F}$-adapted process such that $Y_0\in D_0$. We assume that there exists a càdlàg and $\mathbb{F}$-adapted process $A$ such $A\in\mathring{\mathcal{D}}$ and $\inf_{t\in[0,T]}d(A_t,\partial D_t)>0$. Then, there exists a unique solution to the problem $SP_{\mathcal{D}}(Y)$.
	\end{proposition}
	\begin{proof}
		The proof is similar to the one given in the previous section, as the integral in the reflection problem is in Stieltjes sense. The proof  in the  previous section shows  that if $Y$ is  càdlàg and  $\mathbb{F}$-adapted, then $X$ and $K$ are also càdlàg and  $\mathbb{F}$-adapted processes.
	\end{proof}
	In the following, we provide an estimate for the solution for inputs that are semimartingales.	\begin{proposition}\label{estim_quad}
		Let  $\mathcal{D}=\{ D_t, t\in[0,T]\}$ be a family of time-dependent sets such that for every $t$, $D_t\in\mathcal{C}$ and $t\to D_t$ is càdlàg with respect to Hausdorff metric. Let $Y$ and $\tilde{Y}$ be two semimartingales with the following decompositions: $Y=M+V$ and $\tilde{Y}=\tilde{M}+\tilde{V}$, such that $Y_{0},\tilde{Y}_{0}\in D_0$. Assume that $M$ and $\tilde{M}$ are martingales in $\mathcal{S}_{0,T}$, $V$ and $\tilde{V}$ are càdlàg adapted processes with finite variation such that $|V|$ and $|\tilde{V}|$ are square-integrable.
		Let $(X,K)$  and  $(\tilde{X},\tilde{K})$ be the solutions associated with $RP_{\mathcal{D}}(Y)$ and $RP_{\mathcal{D}}(\tilde{Y})$, respectively. Then, there exists a constant $C$ such that for all stopping times $\tau$ in $[0,T]$,  we have:
		\begin{equation*}
			\mathbb{E}(\sup_{0\leq t\leq \tau}|X_{t}-\tilde{X}_{t}|^{2})\leq C\mathbb{E}\left(|X_0-\tilde{X}_{0}|^2+[M-\tilde{M}]_{\tau}+|V-\tilde{V}|_{\tau}^2\right),
		\end{equation*}
		where $[M-\tilde{M}]$ denotes the quadratic variation of $M-\tilde{M}$.
	\end{proposition}
	\begin{proof}
		By Lemma \ref{estim}, we have 
		\begin{equation*}
			|X_{t}-\tilde{X}_{t}|^2\leq |Y_t-\tilde{Y}_{t}|^2 + 2\int_{0}^{t}\langle Y_t-Y_s+\tilde{Y}_s-\tilde{Y}_t,\mathrm{d}(K_s-\tilde{K}_s) \rangle.
		\end{equation*}
		We apply It\^o's formula, we get
		\begin{align*}
			\int_{0}^{t}\langle Y_t-Y_s+\tilde{Y}_s-\tilde{Y}_t,\mathrm{d}(K_s-\tilde{K}_s)&=\int_{0}^{t}\langle K_{s^-}-\tilde{K}_{s^-},\mathrm{d}(Y_s-\tilde{Y}_s)\rangle\\
			&= \int_{0}^{t}\langle X_{s^-}-\tilde{X}_{s^-},\mathrm{d}(Y_s-\tilde{Y}_s)\rangle-\int_{0}^{t}\langle Y_{s^-}-\tilde{Y}_{s^-},\mathrm{d}(Y_s-\tilde{Y}_s)\rangle\\
			&=\int_{0}^{t}\langle X_{s^-}-\tilde{X}_{s^-},\mathrm{d}(Y_s-\tilde{Y}_s)\rangle + [Y-\tilde{Y}]_{t}+ |Y_0-\tilde{Y}_{0}|^2 - |Y_t-\tilde{Y}_{t}|^2
		\end{align*}
		We use the decomposition of $Y$ and $\tilde{Y}$, apply Burkholder-Davis-Gundy's inequality, and perform a simple computation to obtain
		\begin{align*}
			\mathbb{E}(\sup_{0\leq t\leq \tau}|X_{t}-\tilde{X}_{t}|^{2})
			&\leq  \frac{1}{2} \mathbb{E}(\sup _{s \leq \tau}|X_{s}-\hat{X}_{s}|^2)+ C\mathbb{E}\left(|Y_0-\tilde{Y}_{0}|^2+[M-\tilde{M}]_{\tau}+ |V-\tilde{V}|_{\tau}^2\right).
		\end{align*}
		This completes the proof.
	\end{proof}
	\begin{remark}\label{Rq}
		Note that if for some stopping time $\tau$ in $[0,T]$, we have $$\mathbb{E}(\sup_{t \in[0,\tau]}|A_t|^2)+\mathbb{E}(\sup_{t \in[0,\tau]}|Y_t|^2)<\infty.$$ From inequality (\ref{estim_k_}), we also have	$$\mathbb{E}(|K|_{\tau} + \sup_{t \in[0,\tau]}|X_t|^2)<\infty.$$ 
	\end{remark}
	\subsection{Existence and uniqueness of strong solution for reflected Mckean-Vlasov equations}
	We first recall the existence and uniqueness result in the case of SDEs.
	\begin{definition}\label{def_EDS}
		Let  $\mathcal{D}=\{ D_t, t\in[0,T]\}$ be a family of time-dependent sets such that for every $t$,  $D_t\in\mathcal{C}$ and $t\to D_t$ is càdlàg with respect to Hausdorff metric. Let $X_{0}$  be an $\mathcal{F}_{0}$ measurable random variable, with $ X_{0}\in D_{0}$, 
		$(b, \sigma):[0, T] \times \Omega \times \mathbb{R}^d \times \mathcal{P}\left(\mathbb{R}^d\right) \rightarrow \mathbb{R}^d \times \mathbb{R}^{d \times m}$ and $\beta:[0, T] \times \Omega \times \mathbb{R}^d  \times \mathcal{P}\left(\mathbb{R}^d\right) \times \mathbb{R}^d \backslash\{0\} \rightarrow \mathbb{R}^d $ be measurable random maps.
		A couple $(X, K)$ is said to be a solution to the reflected McKean-Vlasov SDE, which we denote by $E(\mathcal{D}, b,\sigma,\beta)$, if $(X,K)=RP_{\mathcal{D}}(Y)$, where $Y$ is given by the following:
		\begin{equation*}
			Y_t= X_{0}+ \int_{0}^{t} b\left(s, X_s, \mathcal{L}_{X_s}\right) \mathrm{d} s+ \int_{0}^{t} \sigma\left(s, X_s, \mathcal{L}_{X_s}\right) \mathrm{d} W_s+\int_{0}^{t}\int_{\mathbb{R}^d \backslash\{0\}}\beta\left(s,  X_{s^-}, \mathcal{L}_{X_{s^-}},z\right)  \tilde{N}(\mathrm{d} s, \mathrm{d} z),\,\text{for}\;\; t\geq 0.
		\end{equation*}
	\end{definition}
	We consider the following set of assumptions:
	\begin{itemize}
		\item $\left(\mathbf{A_1}\right)$: $\mathcal{D}=\{ D_t, t\in[0,T]\}$ is a family of time-dependent sets such that for every $t$,  $D_t\in\mathcal{C}$ and $t\to D_t$ is càdlàg with respect to Hausdorff metric.
		\item$\left(\mathbf{A_2}\right)$: There exists a process $A$ in $\mathcal{S}_{0,T}$ such that $A\in\mathring{\mathcal{D}}$, and $\inf_{t\in[0,T]}d(A_t,\partial D_t)>0$.
		\item $\left(\mathbf{A_3}\right)$: $X_{0}$  is $\mathcal{F}_{0}$ measurable random variable with $ X_{0}\in D_{0}$ and $\mathbb{E}(|X_0|^2)<\infty$.
		\item $\left(\mathbf{A_4}\right)$:  For fixed $\mu\in \mathcal{P}(\mathbb{R}^d)$, the processes $b(.,.,0,\mu)$ and $\sigma(.,.,0,\mu)$ are  elements of $\mathbb{H}^{2,d}$ and $\mathbb{H}^{2,dm}$, respectively. In addition, for $(\omega,t)$ fixed, there exists $\gamma>0$ such that $$\quad\left|(b, \sigma)(t, \omega, x, \mu)-(b, \sigma)\left(t, \omega, x^{\prime}, \mu^{\prime}\right)\right| \leq \gamma\left(\left|x-x^{\prime}\right|+W\left(\mu, \mu^{\prime}\right)\right).$$
		\item $\left(\mathbf{A_5}\right)$: $\beta$ is $\mathscr{P} \otimes \mathscr{B}_{\mathbb{R}^d}\otimes\mathcal{P}_{2}(\mathbb{R}^d) \otimes \mathcal{B}_{\mathbb{R}^d \backslash\{0\}}$-measurable. For every $(t,\omega)$, we have
		$$\left(\int_{\mathbb{R}^d \backslash\{0\}}\left|\beta(s,\omega,x,\mu, z)-\beta\left(s,\omega,x^{\prime},\mu^{\prime}, z\right)\right|^2\lambda(\mathrm{d} z)\right)^\frac{1}{2} \leq \gamma \left(\left|x-x^{\prime}\right|+W\left(\mu, \mu^{\prime}\right)\right),$$
		$$\int_{\mathbb{R}^d \backslash\{0\}} |\beta(s,\omega,x,\mu, z)|^2\lambda(\mathrm{d} z)\leq \gamma(1+|x|^2).$$
	\end{itemize}
	
	\begin{theorem}\label{theorem_1}
		Under assumptions $\left(\mathbf{A}_1\right)$-$\left(\mathbf{A}_5\right)$, for every $\mu \in\mathcal{P}_2(\mathbb{D}([0,T],\mathbb{R}^d))$, there exists a unique solution $(X^{\mu},K^{\mu})=RP_{\mathcal{D}}(Y^{\mu})$, where $X^{\mu}\in\mathcal{S}^{2}_{0,T}$, and $Y^{\mu}$ is given by: 
		\begin{equation}
			Y_t^{\mu}= X_{0}+ \int_{0}^{t} b\left(s, X_s^{\mu}, \mu_s\right) \mathrm{d} s+ \int_{0}^{t} \sigma\left(s, X_s^{\mu}, \mu_s\right) \mathrm{d} W_s+\int_{0}^{t}\int_{\mathbb{R}^d\backslash\{0\}}\beta\left(s,  X_{s^-}^{\mu}, \mu_{s^-},z\right)  \tilde{N}(\mathrm{d} s, \mathrm{d} z),\;\;\text{for}\;\; t\geq 0.\label{eq1}
		\end{equation}
	\end{theorem}
	\begin{proof}
		Note that for $\mu\in\mathcal{P}_2(\mathbb{D}([0,T],\mathbb{R}^d))$ and $s,t\geq 0$ we have
		$$W\left(\mu_s, \mu_t\right)^2 \leq \int_{\mathbb{D}([0,T],\mathbb{R}^d)}\left|\pi_s(\omega)-\pi_t(\omega)\right|^2 \mu(\mathrm{d} \omega),\; \text{and}\;
		W\left(\mu_s, \mu_{t^-}\right)^2 \leq \int_{\mathbb{D}([0,T],\mathbb{R}^d)}\left|\pi_s(\omega)-\pi_{t^-}(\omega)\right|^2 \mu(\mathrm{d} \omega).$$
		By Lebesgue’s dominated convergence theorem, we conclude that $t\to \mu_t$ is càdlàg with respect to Wasserstein metric.
		Hence, by $\left(\mathbf{A}_4\right)$ and $\left(\mathbf{A}_5\right)$, the stochastic integrals in \eqref{eq1} are well-defined.
		
		Let $\alpha \in ]0,1[$, we consider the following stopping time:
		\begin{equation*}
			\tau:=\inf\{t>0 \;:\;  C\gamma^2(2t+t^2)\geq \alpha\}\land T,
		\end{equation*}
		with $ \tau=\infty\; \text{for} \;\inf(\emptyset)$, and $C$ is the constant in Proposition \ref{estim_quad}.
		
		We consider the mapping $\varphi \colon\mathcal{S}^{2}_{0,\tau}\to\mathcal{S}_{0,\tau}^{2} $ that associates $X\in\mathcal{S}_{0,\tau}^{2}$ with $\varphi(X) $, where $\varphi(X) $ is defined as the first coordinate of the solution of the reflection problem $RP_{\mathcal{D}}\Big(Y^\mu_{.\land\tau} \Big)$, which is justified by Proposition \ref{sk_stoch}.
		
		We have for every $X\in \mathcal{S}^{2}_{0,\tau}$, $\varphi(X)\in \mathcal{S}^{2}_{0,\tau}$. Indeed, 
		from assumptions $\left(\mathbf{A}_4\right)$, $\left(\mathbf{A}_5\right)$, and through the application of Burkholder-Davis-Gundy inequality, we get
		\begin{align*}
			|| Y^{\mu}||_{\mathcal{S}^{2}_{0,\tau}}^2&\leq C_{T,\gamma,\mu} \mathbb{E}\left(|X_0|^2+\int_{0}^{\tau} |X_s|^2 \mathrm{d} s+\int_{0}^{\tau} |b(s,0,\mu_s)|^2\mathrm{d}s+\int_{0}^{\tau} |\sigma(s,0,\mu_s)|^2\mathrm{d}s\right)\\
			&+C_{T,\gamma,\mu}\mathbb{E}\left((1+\sup_{0\leq t\leq \tau}|X_{t}|^{2})\right)\\
			&\leq C_{T,\gamma,\mu}\left(\mathbb{E}\left(|X_0|^2 \right) +||X||_{\mathcal{S}^{2}_{0,\tau}}^2+1\right).
		\end{align*}
		The value of the constant $C_{T,\gamma,\mu}$ depend on $T$, $\gamma$, and $\alpha$. It varies from one line to another.
		Given that $X\in \mathcal{S}^{2}_{0,\tau}$ and based on assumptions $\left(\mathbf{A}_2\right)$ and $\left(\mathbf{A}_3\right)$, as well as the definition of the stopping time $\tau$,  we can conclude, as indicated by Remark \ref{Rq}, that  $||\varphi(X^{\mu})||_{\mathcal{S}^{2}_{0,\tau}}<\infty$.
		
		Now, we show that $\varphi$ is a contraction. Let $X,\tilde{X}\in \mathcal{S}^{2}_{0,\tau}$, by proposition \ref{estim_quad}, we have
		\begin{align*}
			|| \varphi(X)-\varphi(\tilde{X}) ||_{\mathcal{S}^{2}_{0,\tau}}^2&\leq C\Big(\mathbb{E}(\int_{0}^{\tau} |\sigma\left(s, X_s, \mu_s\right)-\sigma(s, \tilde{X}_s, \mu_s)|^2 \mathrm{d} s)\\
			&+\mathbb{E}( \int_{0}^{\tau} \int_{\mathbb{R}^d\backslash\{0\}}|\beta(s, X_{s^-}, \mu_{s^-},z)-\beta(s, \tilde{X}_s, \mu_{s^-},z)|^2 \lambda(\mathrm{d} z)\mathrm{d} s)\\
			&+\mathbb{E}(\tau \int_{0}^{\tau} |b(s, X_s, \mu_s)-b(s, \tilde{X}_s, \mu_s)|^2 \mathrm{d} s\Big)\\
			& \leq C\gamma^2 \mathbb{E}\left((2\tau+\tau^2) \sup_{0\leq t\leq \tau}|X_{t}-\tilde{X}_{t}|^{2}\right)\\
			&\leq \alpha ||X-\tilde{X}||_{\mathcal{S}^{2}_{0,\tau}}^2.\\
		\end{align*}
		Therefore, by the Banach fixed point theorem, we have  existence and uniqueness of solution on $[0,\tau]$.
		
		By induction, we define the following sequence of stopping times:
		$\tau_{0}=0,\;\text{ and }\; n\geq 1$
		$$
		\tau_{n+1}:=\inf\{t>\tau_{n}:\; C\gamma^2(2(t-\tau_{n})+t^2-\tau_{n}^2)\geq \alpha\}\land T.
		$$
		Using a similar argument as previously discussed, one can demonstrate the existence and uniqueness of the solution on each interval $[\tau_{n},\tau_{n+1}]$. Since $t\to C\gamma^2(2(t-\tau_{n})+t^2-\tau_{n}^2)$ is continuous, we obtain $\mathbb{P}(\cap_{n\geq1}\{\tau_{n}=T\})=1$, which implies the existence and uniqueness of a solution on  $[0,T]$.
		
		We show that $X^{\mu}\in\mathcal{S}^{2}_{0,T}$. Given assumptions  $\left(\mathbf{A}_4\right)$ and $\left(\mathbf{A}_5\right)$, along with inequality \ref{estim_k_}, we find
		\begin{align*}
			|| X^{\mu}||_{\mathcal{S}^{2}_{0,T}}^2&\leq C_{T,\gamma} \mathbb{E}\left(|X_0|^2+\int_{0}^{T} (1+\sup_{0 \leq s \leq u}|X_u|^2) \mathrm{d} s+\int_{0}^{T} |b(s,0,\mu_s)|^2\mathrm{d}s+\int_{0}^{T} |\sigma(s,0,\mu_s)|^2\mathrm{d}s +|| A||_{\mathcal{S}^{2}_{0,T}}\right).\\
		\end{align*}
		We utilize Gronwall's Lemma, and we conclude based on assumptions  $\left(\mathbf{A}_2\right)$ and $\left(\mathbf{A}_3\right)$.
		This completes the proof.
	\end{proof}
	
	\begin{theorem}\label{Main_th}
		Under assumptions $\left(\mathbf{A}_1\right)$-$\left(\mathbf{A}_5\right)$, a unique solution exists for the system $E(\mathcal{D}, b,\sigma,\beta)$.
	\end{theorem}
	\begin{proof}
		By Theorem \ref{theorem_1}, for every $\mu \in\mathcal{P}_2(\mathbb{D}([0,T],\mathbb{R}^d))$, there exists a unique solution $(X^{\mu},K^{\mu})$
		satisfying (ii), (iii) of Definition \ref{stoch_def}  and the following equation
		\begin{equation*}
			X_t^{\mu}= X_{0}+ \int_{0}^{t} b\left(s, X_s^{\mu}, \mu_s\right) \mathrm{d} s+ \int_{0}^{t} \sigma\left(s, X_s^{\mu}, \mu_s\right) \mathrm{d} W_s+\int_{0}^{t}\int_{\mathbb{R}^d\backslash\{0\}}\beta\left(s,  X_{s^-}^{\mu}, \mu_{s^-},z\right)  \tilde{N}(\mathrm{d} s, \mathrm{d} z)+K_t,\;\;\text{for}\; t\geq 0.
		\end{equation*}
		We define the mapping $\psi\colon\mathcal{P}_2(\mathbb{D}([0,T],\mathbb{R}^d))\to\mathcal{P}_2(\mathbb{D}([0,T],\mathbb{R}^d))$ which associates $\mu\in\mathcal{P}_2(\mathbb{D}([0,T],\mathbb{R}^d))$ with $\psi(\mu)=\mathcal{L}(X^{\mu})$. We show that $\psi$ is a contraction.
		
		Consider $\mu,\mu^{\prime}\in\mathbb{D}([0,T],\mathbb{R}^d)$. By Proposition \ref{estim_quad} and under assumptions $\left(\mathbf{A}_4\right)$ and $\left(\mathbf{A}_5\right)$, for every $t$, we have
		\begin{align*}
			\mathbb{E}(\sup_{0\leq s\leq t}|X_{s}^{\mu}-X_{s}^{\mu^{\prime}}|^{2})&\leq C\gamma^2 (2+t)\left(\int_{0}^{t}\mathbb{E}(\sup_{0\leq u\leq s}|X_{u}^{\mu}-X_{u}^{\mu^{\prime}}|^{2})+W(\mu_s, \mu^{\prime}_s)^2\mathrm{d}s\right)\\
			& + \int_{0}^{t} W(\mu_{s^-}, \mu_{s^-}^{\prime})^2 \mathrm{d} s \\
			&\leq 2C\gamma^2 (2+T)\left(\int_{0}^{t}\mathbb{E}(\sup_{0\leq u\leq s}|X_{u}^{\mu}-X_{u}^{\mu^{\prime}}|^{2}) \mathrm{d} s+ \int_{0}^{T} W_s(\mu, \mu^{\prime})^2 \mathrm{d}s\right).
		\end{align*}
		In the final line, we have made use of the fact that $W(\mu_s, \mu^{\prime}_s)\lor W(\mu_{s^-}, \mu^{\prime}_{s^-})\leq W_s(\mu, \mu^{\prime})$. By applying Gronwall’s inequality and considering the definition of the Wasserstein metric, we obtain
		\begin{equation*}
			W_T\left(\psi(\mu), \psi(\mu^{\prime})\right)^2\leq  C_{T,\gamma} \int_{0}^{T} W_s(\mu, \mu^{\prime})^2 \mathrm{d}s.
		\end{equation*}
		Here, $C_{T,\gamma}$ is a constant depend on $T$ and $\gamma$.
		By iterating the previous inequality, we get for any $n\geq 1$
		
		$$	W_T\left(\psi^n(\mu), \psi^n\left(\mu^{\prime}\right)\right)^2\leq \frac{C_{T,\gamma}^n T^n}{n !} W_T\left(\mu, \mu^{\prime}\right)^2,$$
		for $n$ large enough,  $\psi$ is a contraction. 
		
		This completes the proof.
	\end{proof}
	\subsection{Stability of Reflected Mckean-Vlasov equations}
	In this subsection, we study a stability result for the solution of the system $E(\mathcal{D},b,\sigma,\beta)$ with respect to both the initial condition and the coefficients.
	\begin{theorem}
		Assume that $\left(\mathbf{A}_1\right)$-$\left(\mathbf{A}_2\right)$ are satisfied. 
		Let $(X_0^n)_{n\geq 0}$, $X_0$ be sequence of random variables satisfying  $\left(\mathbf{A}_3\right)$. Let  $b,\sigma,\beta,(b^n)_{n\geq 0},(\sigma^n)_{n\geq 0},(\beta^n)_{n\geq 0}$ be sequence of random maps satisfying $\left(\mathbf{A}_4\right)$-$\left(\mathbf{A}_5\right)$ uniformly in $n$. In addition, we assume the following:
		\begin{itemize}
			\item[i)] $\lim _{n \rightarrow \infty}\mathbb{E}|X^n_0-X_0|^2=0$.
			\item[ii)]  There exists a  positive constant $\delta$, such that
			$$
			\sup _{n \geq 0}\left(\left|b^n(t,x,\mu)\right|+\left|\sigma^n(t,x,\mu)\right| \right)\leq \delta\left(1+|x|\right) .
			$$
			\item[iii)]  On each compact $K$ of $\mathbb{R}^d$, we have
			$$\lim _{n \rightarrow \infty} \sup _{0\leq t \leq T} \sup _{x \in K} \sup _{\mu \in \mathcal{P}_2\left(\mathbb{R}^d\right)}\left|b^n(t, x, \mu)-b(t, x, \mu)\right|+\left|\sigma^n(t, x, \mu)-\sigma(t, x, \mu)\right|=0,\;\;\text{a.s}$$
			$$ \lim _{n \rightarrow \infty} \sup _{0\leq t \leq T} \sup _{x \in K} \sup _{\mu \in \mathcal{P}_2\left(\mathbb{R}^d\right)}\int_{\mathbb{R}^d \backslash\{0\}}\left|\beta^n(t, x,\mu,z)-\beta(t, x,\mu,z)\right|^2\lambda(\rm{d}z)=0,\;\;\text{a.s}$$ 
		\end{itemize}
		then
		$$
		\lim _{n \rightarrow \infty} E\left(\sup _{0\leq t \leq T}\left|X_t^n-X_t\right|^2\right)=0,
		$$
		where $(X^n,K^n)$ is sequence of solutions of $E_{\mathcal{D}}(b^n,\sigma^n,\beta^n)$, and $(X,K)$ is the solution of $E_{\mathcal{D}}(b,\sigma,\beta)$.
	\end{theorem}
	\begin{proof}
		We use Proposition \ref{estim_quad}, we get
		
		\begin{align*}
			\mathbb{E}(\sup_{0 \leq t \leq T}|X_t^{n}-X_t|^2)&\leq 
			C\Big(\mathbb{E}(|X_0^{n}-X_0|^2+\int_{0}^{T} |\sigma^n\left(s, X_s^{n}, \mathcal{L}_{X_s^n}\right)-\sigma\left(s, X_s, \mathcal{L}_{X_s}\right)|^2 \mathrm{d} s)\\
			&+\mathbb{E}( \int_{0}^{T} \int_{\mathbb{R}^d \backslash\{0\}}|\beta^n(s, X_{s^-}^{n}, \mathcal{L}_{X_{s^-}^n},z)-\beta(s, X_s
			,\mathcal{L}_{X_{s^-}}, z)|^2 \lambda(\mathrm{d} z)\mathrm{d} s)\\
			&+\mathbb{E}(T\int_{0}^{T} |b^n\left(s, X_s^{n}, \mathcal{L}_{X_s^n}\right)-b\left(s, X_s, \mathcal{L}_{X_s}\right)|^2 \mathrm{d} s)\Big).
		\end{align*}
		We add and subtract the quantities $\sigma^n\left(s, X_s, \mathcal{L}_{X_s}\right), \beta^n\left(s, X_s, \mathcal{L}_{X_{s^-}},z\right)$, and $b^n\left(s, X_s, \mathcal{L}_{X_s}\right)$, by using the fact that the coefficients are Lipschitz, and from the definition of Wassertein metric,  we get
		\begin{align*}
			\mathbb{E}\left(\sup_{0 \leq t \leq T}|X_t^{n}-X_t|^2\right)&\leq C_{T,\gamma} \Big(\int_0^T \mathbb{E}\left(\sup _{0\leq u \leq s}\left|X_u^n-X_u\right|^2\right) \mathrm{d}s+\mathbb{E}(\left|X_0^{n}-X_0|^2\right)\\
			&+ \mathbb{E}(\int_{0}^{T} |\sigma^n\left(s, X_s, \mathcal{L}_{X_s}\right)-\sigma\left(s, X_s, \mathcal{L}_{X_s}\right)|^2 \mathrm{d} s)\\
			&+\mathbb{E}( \int_{0}^{T} \int_{\mathbb{R}^d \backslash\{0\}}|\beta^n(s, X_{s^-}, \mathcal{L}_{X_{s^-}},z)-\beta(s, X_s
			,\mathcal{L}_{X_{s^-}}, z)|^2 \lambda(\mathrm{d} z)\mathrm{d} s)\\
			&+\mathbb{E}(\int_{0}^{T} |b^n\left(s, X_s, \mathcal{L}_{X_s}\right)-b\left(s, X_s, \mathcal{L}_{X_s}\right)|^2 \mathrm{d} s)\Big).
		\end{align*}
		By applying Gronwall's lemma, we obtain
		
		\begin{align*}
			\mathbb{E}\left(\sup_{0 \leq t \leq T}|X_t^{n}-X_t|^2\right)&\leq C_{T,\gamma} \Big(\mathbb{E}(|X_0^{n}-X_0|^2)+ \mathbb{E}(\int_{0}^{T} |\sigma^n\left(s, X_s, \mathcal{L}_{X_s}\right)-\sigma\left(s, X_s, \mathcal{L}_{X_s}\right)|^2 \mathrm{d} s)\\
			&+\mathbb{E}( \int_{0}^{T} \int_{\mathbb{R}^d \backslash\{0\}}|\beta^n(s, X_{s^-}, \mathcal{L}_{X_{s^-}},z)-\beta(s, X_s
			,\mathcal{L}_{X_{s^-}}, z)|^2 \lambda(\mathrm{d} z)\mathrm{d} s)\\
			&+\mathbb{E}(\int_{0}^{T} |b^n\left(s, X_s, \mathcal{L}_{X_s}\right)-b\left(s, X_s, \mathcal{L}_{X_s}\right)|^2 \mathrm{d} s)\Big).
		\end{align*}
		The first term of this inequality tend to $0$ by $(i)$. From the assumptions $\left(\mathbf{A}_5\right)$, $(ii)$, and $(iii)$, the Dominated Convergence Theorem
		guarantees that the three last terms tend to $0$. 
		
		This completes the proof.
	\end{proof}
	
	\subsection{Propagation of Chaos }
	Let $N\geq 1$ and $i\in\{1,\dots,N\}$. We consider $(X_{0}^i,W^i,\tilde{N}^i)$ as a sequence of independent copies of $(X_0,W,\tilde{N})$. In this subsection, we study the law of a solution to the interacting $N$-particle system $(X^{N,i},K^{N,i})_{1\leq i\leq N}$ given by the following
	
	$\textbf(S)\left\{\begin{aligned}
		&X_t^{N,i}= X_{0}^i+ \int_{0}^{t} b\left(s, X_s^{N,i}, \frac{1}{N}\sum_{j=1}^N\delta_{X_s^{j,N}}\right) \mathrm{d} s+ \int_{0}^{t} \sigma\left(s, X_s^{N,i}, \frac{1}{N}\sum_{j=1}^N\delta_{X_s^{j,N}}\right) \mathrm{d} W_s^i \\
		& \quad\quad		+\int_{0}^{t}\int_{\mathbb{R}^d\backslash\{0\}}\beta\left(s,  X_{s^-}^{N,i}, \frac{1}{N}\sum_{j=1}^N\delta_{X_{s^-}^{j,N}},z\right)  \tilde{N}^i(\mathrm{d} s, \mathrm{d} z)+ K_t^{i,N}, \quad t \in[0, T], \\
		&
		X_t^{i,N}\in D_t, \quad \text{for every}\; t \in[0, T], \\
		& \text{for every  càdlàg and adapted process} \;Z\in\mathcal{D},\; \int_0^t \langle X_{s}^{i,N} - Z_s,  \mathrm{d}K_s^{i,N} \rangle \leq 0\;\text{for every}\;t\geq 0.
	\end{aligned}\right.$
	\begin{proposition}
		Under assumptions $\left(\mathbf{A}_1\right)$-$\left(\mathbf{A}_5\right)$, a unique solution exists for the system $(\textbf{S})$.
	\end{proposition}
	\begin{proof}
		Since for every $u=(u_1,\dots,u_N)$ and $v=(v_1,\dots,v_N)$ in $\mathbb{R}^{dN}$, $$W\left(\frac{1}{N} \sum_{i=1}^N \delta_{u_i}, \frac{1}{N} \sum_{i=1}^N \delta_{v_i}\right) \leq\frac{1}{\sqrt{N}}|u-v|,$$ the proof follows similarly as the proof of Theorem \ref{Main_th}.
	\end{proof}
	\begin{theorem}
		Let $(X^{i},K^{i})$ be the solution of $RP(Y^i)_{\mathcal{D}}$, where $Y^i$ is given by:
		$$Y_t^i=X_{0}^i+ \int_{0}^{t} b\left(s, X_s^i, \mathcal{L}_{X_s^i}\right) \mathrm{d} s+ \int_{0}^{t} \sigma\left(s, X_s^i, \mathcal{L}_{X_s^i}\right) \mathrm{d} W_s^i+\int_{0}^{t}\int_{\mathbb{R}^d\backslash\{0\}}\beta\left(s,  X_{s^-}^i, \mathcal{L}_{X_{s^-}^i},z\right)  \tilde{N}^i(\mathrm{d} s, \mathrm{d} z).$$ 
		Then, there exists a constant $C$, which depend on $T$, such that
		$$	\mathbb{E}(\sup_{0 \leq t \leq T}|X_t^{i,N}-X_t^i|^2) \leq C \begin{cases}N^{-1 / 2}, & d<4 \\ N^{-1 / 2} \log N, & d=4 \\ N^{-\frac{2}{d}}, & d>4\end{cases}$$
		Consequently, the following convergence hold true:
		$$
		\lim _{N \rightarrow \infty} \sup _{1 \leq i \leq N} \mathbb{E} \sup _{0 \leq t \leq T}|X_t^{i, N}-X_t^i|^2=0 .
		$$
		And for all $k\geq 1$, the following weak convergence holds true:
		$$
		\mathcal{L}_{\left(X^{N, 1}, X^{N, 2}, \ldots, X^{N, k}\right)} \underset{N \rightarrow \infty}{\longrightarrow} \mathcal{L}_{\left(X^1,X^2, \ldots,X^k\right)}.
		$$
	\end{theorem}
	\begin{proof}Set $$\mu^N= \frac{1}{N}\sum_{j=1}^N\delta_{X^{j,N}},\; \text{and} \;\mu= \frac{1}{N}\sum_{j=1}^N\delta_{X^{j}}.$$ By applying the inequality in Proposition \ref{estim_quad}, and by adding and subtracting the terms $\sigma\left(s, X_s^i, \mu_s^N\right),\; \sigma\left(s, X_s^i, \mu_s\right)$, $\beta\left(s, X_{s^-}^i, \mu_{s^-}^N,z\right),\;\beta\left(s, X_{s^-}^i, \mu_{s^-},z\right)$, $b\left(s, X_s^i, \mu_s^N\right)$, and $b\left(s, X_s^i, \mu_s\right)$, we obtain
		\begin{align*}
			\mathbb{E}(\sup_{0 \leq t \leq T}|X_t^{i,N}-X_t^i|^2)&\leq 
			C\Big(\mathbb{E}(\int_{0}^{T} |\sigma\left(s, X_s^{i,N}, \mu_s^N\right)-\sigma\left(s, X_s^i, \mu_s^N\right)|^2 \mathrm{d} s)\\
			&+\mathbb{E}(\int_{0}^{T} |\sigma\left(s, X_s^{i}, \mu_s^N\right)-\sigma(s, X_s^i, \mu_s)|^2 \mathrm{d} s)\\
			&+\mathbb{E}(\int_{0}^{T} |\sigma\left(s, X_s^{i}, \mu_s\right)-\sigma(s, X_s^i, \mathcal{L}_{X^i_s})|^2 \mathrm{d} s)\\
			&+\mathbb{E}( \int_{0}^{T} \int_{\mathbb{R}^d\backslash\{0\}}|\beta(s, X_{s^-}^{i,N}, \mu_{s^-}^N,z)-\beta(s, X_{s^-}^{i}, \mu_{s^-}^N,z)|^2 \lambda(\mathrm{d} z)\mathrm{d} s)\\
			&+\mathbb{E}( \int_{0}^{T} \int_{\mathbb{R}^d\backslash\{0\}}|\beta(s, X_{s^-}^{i}, \mu_{s^-}^N,z)-\beta(s, X_{s^-}^{i}, \mu_{s^-},z)|^2 \lambda(\mathrm{d} z)\mathrm{d} s)\\
			&+\mathbb{E}( \int_{0}^{T} \int_{\mathbb{R}^d\backslash\{0\}}|\beta(s, X_{s^-}^{i}, \mu_{s^-},z)-\beta(s, X_{s^-}^{i}, \mathcal{L}_{X^i_{s^-}},z)|^2 \lambda(\mathrm{d} z)\mathrm{d} s)\\
			&+\mathbb{E}(T\int_{0}^{T} |b\left(s, X_s^{i,N}, \mu_s^N\right)-b\left(s, X_s^i, \mu_s^N\right)|^2 \mathrm{d} s)\\
			&+\mathbb{E}(T\int_{0}^{T} |b\left(s, X_s^{i}, \mu_s^N\right)-b(s, X_s^i, \mu_s)|^2 \mathrm{d} s)\\
			&+\mathbb{E}(T\int_{0}^{T} |b\left(s, X_s^{i}, \mu_s\right)-b(s, X_s^i, \mathcal{L}_{X^i_s})|^2 \mathrm{d} s)\Big).\\
		\end{align*}
		Using assumptions $\left(\mathbf{A}_4\right)$ and $\left(\mathbf{A}_5\right)$, along with Gronwall's Lemma and the exchangeability of $(X^{i,N},X^i)$, we find:
		\begin{align*}
			\mathbb{E}(\sup_{0 \leq t \leq T}|X_t^{i,N}-X_t^i|^2)&\leq C_{T,\gamma}\left(  \mathbb{E}\left(\int_0^T  W(\mu_s,\mathcal{L}_{X_s^i})^2 \rm{d}s\right) + \mathbb{E}\left(\int_0^T  W(\mu_{s^-},\mathcal{L}_{X_{s^-}^i})^2 \rm{d}s\right)\right).
		\end{align*}
		By Theorem \ref{theorem_1}, $X^i\in\mathcal{S}^2_{0,T}$, then $\mu$ and $\mathcal{L}_{X^i}$ belong to $\mathcal{P}_2(\mathbb{D}([0,T],\mathbb{R}^d))$. Similar to the proof of Theorem \ref{theorem_1}, it follows that $t\to \mu_t$ and $t\to\mathcal{L}_{X_t}$ are càdlàg with respect to the Wasserstein metric, hence
		\begin{equation}
			\mathbb{E}(\sup_{0 \leq t \leq T}|X_t^{i,N}-X_t^i|^2)\leq 2TC_{T,\gamma} \sup_{0 \leq t \leq T}\mathbb{E}\left(W(\mu_s,\mathcal{L}_{X_s^i})^2\right). \label{chaos}
		\end{equation}
		By Theorem 5.8 in \cite{Delarue}, there exists a positive constant $C$ such that
		\begin{equation}
			\mathbb{E}\left[W\left(\mu_s, \mathcal{L}_{X_s^i}\right)^2\right] \leq C \begin{cases}N^{-1 / 2}, & d<4, \\ N^{-1 / 2} \log N, & d=4, \\ N^{-\frac{2}{d}}, & d>4.\end{cases} \label{ineq-chaos}\end{equation}
		The second point follows immediately from this inequality.
		
		For the last point, we use the second point of Remark \ref{rq2}, the definition of Wasserstein metric and the exchangeability of $(X^{i,N},X^i)$,  we get
		\begin{align*}
			W_{d_{\mathbb{S}}}\left(\mathcal{L}_{\left(X^{N, 1}, X^{N, 2}, \ldots, X^{N, k}\right)}, \mathcal{L}_{\left(X^1,X^2, \ldots,X^k\right)}\right)^2&\leq W_{
				T}\left(\mathcal{L}_{\left(X^{N, 1}, X^{N, 2}, \ldots, X^{N, k}\right)}, \mathcal{L}_{\left(X^1,X^2, \ldots,X^k\right)}\right)^2\\
			& \leq k \mathbb{E}(\sup _{0\leq i \leq k}\left|X^{i, n}-X^i\right|^2 ).
		\end{align*}
		
		Again by remark \ref{rq2}, and  Theorem 6.9 in \cite{Villani}, we get $\mathcal{L}_{\left(X^{N, 1}, X^{N, 2}, \ldots, X^{N, k}\right)}$ converges weakly to $\mathcal{L}_{\left(X^1,X^2, \ldots,X^k\right)}$.
		
		This completes the proof.
	\end{proof}
	
\end{document}